\documentclass[12pt]{amsart}
\usepackage{amscd,amsmath,amsthm,amssymb,graphics}
\usepackage{pstcol,pst-plot,pst-3d}
\usepackage{lmodern,pst-node}
\usepackage{multicol}
\usepackage{epic,eepic}
\usepackage{amsfonts,amssymb,amscd,amsmath,enumerate,verbatim}
\psset{unit=0.7cm,linewidth=0.8pt,arrowsize=2.5pt 4}

\newpsstyle{fatline}{linewidth=1.5pt}
\newpsstyle{fyp}{fillstyle=solid,fillcolor=verylight}
\definecolor{verylight}{gray}{0.97}
\definecolor{light}{gray}{0.9}
\definecolor{medium}{gray}{0.85}
\definecolor{dark}{gray}{0.6}

\def\frk{\frak}               

\def\Phi{{\frk n}}
\def\Phi{{\frk N}}
\def\opn#1#2{\def#1{\operatorname{#2}}} 
\opn\chara{char} \opn\length{\ell} \opn\pd{pd} \opn\rk{rk}
\opn\projdim{proj\,dim} \opn\injdim{inj\,dim} \opn\rank{rank}
\opn\depth{depth} \opn\grade{grade} \opn\height{ht}
\opn\embdim{emb\,dim} \opn\codim{codim}

\opn\Tr{Tr} \opn\bigrank{big\,rank}
\opn\superheight{superheight}\opn\lcm{lcm}
\opn\trdeg{tr\,deg}
\opn\reg{reg} \opn\lreg{lreg} \opn\ini{in} \opn\lpd{lpd}
\opn\size{size}\opn\bigsize{bigsize}
\opn\cosize{cosize}\opn\bigcosize{bigcosize}
\opn\sdepth{sdepth}\opn\sreg{sreg}
\opn\link{link}\opn\fdepth{fdepth}
\opn\div{div} \opn\Div{Div} \opn\cl{cl} \opn\Cl{Cl}
\opn\Spec{Spec} \opn\Supp{Supp} \opn\supp{supp} \opn\Sing{Sing}
\opn\Ass{Ass} \opn\Min{Min}\opn\Mon{Mon} \opn\dstab{dstab} \opn\astab{astab}
\opn\Syz{Syz}
\opn\Ann{Ann} \opn\Rad{Rad} \opn\Soc{Soc}
\opn\Im{Im} \opn\Ker{Ker} \opn\Coker{Coker} \opn\Am{Am}
\opn\Hom{Hom} \opn\Tor{Tor} \opn\Ext{Ext} \opn\End{End}
\opn\Aut{Aut} \opn\id{id}

\opn\nat{nat}
\opn\pff{pf}
\opn\Pf{Pf} \opn\GL{GL} \opn\SL{SL} \opn\mod{mod} \opn\ord{ord}
\opn\Gin{Gin} \opn\Hilb{Hilb}\opn\sort{sort}
\opn\initial{init}
\opn\ende{end}
\opn\height{ht}
\opn\type{type}
\opn\aff{aff} \opn\con{conv} \opn\relint{relint} \opn\st{st}
\opn\lk{lk} \opn\cn{cn} \opn\core{core} \opn\vol{vol}
\opn\link{link} \opn\star{star}\opn\lex{lex}
\opn\gr{gr}

\def\pot#1#2{#1[\kern-0.28ex[#2]\kern-0.28ex]}

\opn\dirlim{\underrightarrow{\lim}}
\opn\inivlim{\underleftarrow{\lim}}
\def\Implies{\ifmmode\Longrightarrow \else
        \unskip${}\Longrightarrow{}$\ignorespaces\fi}
\def\implies{\ifmmode\Rightarrow \else
        \unskip${}\Rightarrow{}$\ignorespaces\fi}
\def\iff{\ifmmode\Longleftrightarrow \else
        \unskip${}\Longleftrightarrow{}$\ignorespaces\fi}

\let\:=\colon
\newtheorem{Theorem}{Theorem}[section]
 \newtheorem{Lemma}[Theorem]{Lemma}
 \newtheorem{Corollary}[Theorem]{Corollary}
 \newtheorem{Proposition}[Theorem]{Proposition}

 \newtheorem{Example}[Theorem]{Example}

\let\epsilon\varepsilon
\let\kappa=\varkappa
\def\qed{\ifhmode\textqed\fi
      \ifmmode\ifinner\quad\qedsymbol\else\dispqed\fi\fi}
\def\textqed{\unskip\nobreak\penalty50
       \hskip2em\hbox{}\nobreak\hfil\qedsymbol
       \parfillskip=0pt \finalhyphendemerits=0}
\def\dispqed{\rlap{\qquad\qedsymbol}}

\opn\dis{dis}
\def\pnt{{\raise0.5mm\hbox{\large\bf.}}}

\opn\Lex{Lex}

\begin{document}
 \title{Unmixedness and arithmetic properties of matroidal ideals}

 \author {Hero Saremi and Amir Mafi*}

\address{Hero Saremi, Department of Mathematics, Sanandaj Branch, Islamic Azad University, Sanandaj, Iran.}
\email{hero.saremi@gmail.com}

\address{A. Mafi, Department of Mathematics, University of Kurdistan, P.O. Box: 416, Sanandaj,
Iran.}
\email{a\_mafi@ipm.ir}


\keywords{Arithmetical rank, Unmixed ideals, Matroidal ideals.\\
* Corresponding author}

\begin{abstract}
Let $R=k[x_1,...,x_n]$ be the polynomial ring in $n$ variables over a field $k$ and $I$ be a matroidal ideal of degree $d$.
In this paper, we study the unmixedness properties  and the arithmetical rank of $I$. Moreover, we show that $ara(I)=n-d+1$.
This answer to the conjecture that made by H. J. Chiang-Hsieh \cite[Conjecture]{C}.

\end{abstract}

\maketitle

\section*{Introduction}
Throughout this paper, we assume that $R=k[x_1,...,x_n]$ is the polynomial ring in $n$ variables over a field $k$ with the maximal ideal $\frak{m}=(x_1,...,x_n)$, $I$ a monomial ideal of $R$ and $G(I)$ the unique minimal monomial generators set of $I$.

 A monomial ideal $I$ generated in a single degree is {\it polymatroidal ideal} when it is satisfying in the following conditions: for all monomials $u,v\in G(I)$ with $\deg_{x_i}(u)>\deg_{x_i}(v)$, there exists an index $j$ such that $\deg_{x_j}(v)>\deg_{x_j}(u)$ and $x_j(u/x_i)\in I$ (see \cite{HH4} or \cite{HH1}). A square-free polymatroidal ideal is called a {\it matroidal ideal}. The product of polymatroidal ideals is again polymatroidal (see \cite[Theorem 5.3]{CH}). In particular each power of a polymatroidal ideal is polymatroidal. Also, $I$ is a polymatroidal ideal if and only if $(I:u)$ is a polymatroidal ideal for all monomial $u$ (see \cite[Theorem 1.1]{BH}). If $I$ is a matroidal ideal of degree $d$, then $\depth(R/J)=d-1$ and $\pd(R/I)=n-d+1$  \cite[Theorem 2.5]{C}. A monomial ideal $I$ is called unmixed if all prime ideals in $\Ass(R/I)$ have the same height. If $I$ is Cohen-Macaulay, i.e., the quotient ring $R/I$ is Cohen-Macaulay, then $I$ is unmixed.

Herzog and Hibi in \cite{HH2} proved that a polymartoidal $I$ is Cohen-Macaulay if and only if $I$ is a principal ideal, a Veronese ideal, or a square-free Veronese ideal. They gave a counter example which is an unmixed matroidal ideal but it is not Cohen-Macauly in the case $n=6$. Herzog and Hibi in \cite{HH2} leave as a problem the classification of all unmixed polymatroidal ideals. After that Vl$\check{a}$doiu in \cite{Vl} studied the unmixed polymatroidal ideals and he showed that an ideal of Veronese type is unmixed if and only if it is Cohen-Macaulay.

The {\it arithmetical rank} of $I$ is defined as follows:\\
$ara(I):=\min\{t\in\mathbb{N}:$ there exist $a_1,...,a_t\in R$ such that $\sqrt{(a_1,...,a_t)}=\sqrt{I}\}$.\\
Note that ideals with the same radical have the same aritmetical rank.
It is well-known that $\height(I)\leq ara(I)\leq\mu(I)$, where $\mu(I)$ is the minimal number of generators of $I$ and $\height$ is the height of $I$.
If $I$ is a square-free monomial ideal in $R$, then by using \cite{L} and the Auslander-Buchsbaum formula we have the following well-known inequalities:\\
$\height(I)\leq\pd(R/I)\leq ara(I)\leq\mu(I)$, where $\pd$ is the projective dimension of $I$.
In particular, $I$ is Cohen-Macaulay if and only if $\height(I)=\pd(R/I)$

An ideal $I$ is called {\it set-theoretic complete intersection}, when $ara(I)=\height(I)$.
Kimura, Terai and Yoshida in \cite{KTY} raised the following question:\\
Let $I$ be a square-free monomial ideal in $R$. When does $ara(I)=\pd(R/I)$ hold? In paricular, suppose tha $R/I$ is Cohen-Macaulay. When is $I$ a set-theoretic complete intersection?

A considerable number of studies have been made on this question (see, for example, \cite{B}, \cite{B1}, \cite{KTY} and \cite{C}).
The above question does not always hold as was shown by \cite{Y}.
Chiang-Hsieh in \cite{C} proved that if $I$ is a matroidal ideal of degree $d$, then $ara(I)=\pd(R/I)$ provided that one of the following conditions holds:\\
$(i)$ $I$ is square-free Veronese; $(ii)$ $I=J_1J_2...J_d$, where each $J_i$ is generated by $h$ distinct variables; $(iii)$ $d=2$.

In the end of her article, she proposed the following conjecture:

{\bf Conjecture:} Let $I$ be a full-supported matroidal ideal of degree $d$. Then $ara(I)=n-d+1$.

The main propose of this note is to study the unmixed properties and arithmetical rank of matroidal ideals.
Also, we give an affirmative answer to the above conjecture.

For any unexplained notion or terminology, we refer the reader to \cite{HH1} and \cite{V}.
Several explicit examples were  performed with help of the computer algebra systems Macaulay2 \cite{GS}.

\section{Unmixed matroidal ideals}
Let $I$ be a monomial ideal of $R$ and $G(I)=\{u_1,...,u_t\}$. Then we set $[n]=\{x_1,...,x_n\}$ and $\supp(I)=\cup_{i=1}^t\supp(u_i)$, where $\supp(u)=\{x_i: u=x_1^{a_1}...x_n^{a_n}, a_i\neq 0\}$. Throughout this paper we assume that all polymatrodal ideals are full-supported, that is, $\supp(I)=[n]$.

\begin{Lemma}\label{L1}
Let $I$ be a matroidal ideal of degree $d$ and let $x,y$ be two variables in $R$ such that $xy\nmid u$ for all $u\in G(I)$. Then $I:x=I:y$. In particular, if $xu\in G(I)$ for some monomial element $u$ of degree $d-1$, then $yu\in G(I)$.
\end{Lemma}

\begin{proof}
Since $xy\nmid u$ for all $u\in G(I)$, we have $(I:y)=((I:y):x)=(I:xy)=((I:x):y)=(I:x)$.
This completes the proof.
\end{proof}

The following result was proved in \cite[Theorem 3.2]{C}, we give an easier proof.
\begin{Theorem}\label{T1}
Let $I$ be a matroidal ideal of degree $d=2$. Then there are subsets $S_1,...,S_m$ of $[n]$ such that the following conditions hold:
\begin{itemize}
\item[(i)] $m\geq 2$ and $|S_i|\geq 1$ for each $i$;
\item[(ii)]$S_i\cap S_j=\emptyset$ for $i\neq j$ and $\cup_{i=1}^mS_i=[n]$;
\item[(iii)] $xy\in I$ if and only if $x\in S_i$ and $y\in S_j$ for $i\neq j$;
\item[(iv)] $xy\notin I$ if and only if $x,y\in S_i$ for some $i$.
\end{itemize}
\end{Theorem}

\begin{proof}
Since $\frak{m}\notin\Ass(R/I)$, we have $I=(I:\frak{m})$. Therefore, there exists $m\leq n$ such that $I=\cap_{i=1}^m(I:x_i)$ is a minimal intersection.
For $1\leq i\leq m$, we consider $S_i=[n]\setminus G(I:x_i)$. Thus $m\geq 2$ and $|S_i|\geq 1$ for each $i$. Hence $(i)$ holds.\\
$(ii)$ Suppose $y\in S_i\cap S_j$ for $i\neq j$. Then $y\notin G(I:x_i)\cup G(I:x_j)$  and so $yx_i, yx_j\notin I$ for $i\neq j$. Therefore $(I:x_i)=(I:x_iy)=(I:y)=(I:x_jy)=(I:x_j)$ and this is a contradiction.  Thus $S_i\cap S_j=\emptyset$ for $i\neq j$. If $y\in\cap_{i=1}^mG(I:x_i)$, then $y\in G(I:x_i)$ for all $1\leq i\leq m$ and so $yx_i\in I$. Therefore $y\in I=\cap_{i=1}^m(I:x_i)$ and this is a contradiction. Hence $\cap_{i=1}^mG(I:x_i)=\emptyset$ and so $\cup_{i=1}^mS_i=[n]$.
Thus $(ii)$ holds.\\
$(iii)$ If $xy\in I$, then by definition of $S_i$ it is clear that $x\in S_i$ and $y\in S_j$ for $i\neq j$.
Conversely, let $x\in S_i$ and $y\in S_j$ for $i\neq j$. Then $xx_i, yx_j\notin I$ for $i\neq j$ and so $(I:x_i)=(I:x)$ and $(I:y)=(I:x_j)$.
If $xy\notin I$, then by Lemma \ref{L1} $(I:x)=(I:y)$. Therefore $(I:x_i)=(I:x_j)$ and this is a contradiction. Thus $(iii)$ holds.\\
$(iv)$ If $x,y\in S_i$ for some $i$, then it is clear that $xy\notin I$. Conversely, if $xy\notin I$ then by Lemma \ref{L1} $(I:x)=(I:y)$ and so $x,y\in S_i$ for some $i$.
This completes the proof.
\end{proof}

In the following, we assume that $m$ is the number of $S_i$ as we use in the Theorem \ref{T1}.
\begin{Corollary}\label{C1}
Let $I$ be a matroidal ideal of degree $d=2$. Then $I$ is an unmixed ideal if and only if $m(n-\height(I))=n$.
In particular, if $n$ is a prime number and $I$ is an unmixed matroidal ideal of degree $d=2$, then $I$ is Cohen-Macaulay.
\end{Corollary}

\begin{proof}
By using Theorem \ref{T1}, $\cup_{i=1}^mS_i=[n]$ and $S_i=[n]\setminus G(I:x_i)$. Therefore $I$ is an unmixed ideal if and only if $|G(I:x_i)|=|G(I:x_j)|=\height(I)$ for all $i\neq j$. Since $|S_i|=n-|G(I:x_i)|$, it follows that if $I$ is unmixed then $m(n-\height(I))=n$.
Conversely, suppose that $m(n-\height(I))=n$. It is clear that $\height(I)\leq |G(I:x_i)|$ for all $i$. Let, by contrary, $I$ is not unmixed. Then there exists $1\leq i\leq m$ such that $\height(I)<|G(I:x_i)|$. Since $\cup_{i=1}^mS_i=[n]$ and $|S_i|=n-|G(I:x_i)|$ for all $i$, we have $n<m(n-\height(I))$ and this is a contradiction. Therefore $I$ is unmixed.\\
If $n$ is a prime number, then $\height(I)=n-1$. Since $\depth(R/I)>0$, by using the Auslander-Buchsbaum formula, it follows that $\height(I)=\pd(R/I)$. Therefore $I$ is Cohen-Macaulay.
\end{proof}

It is know that for $n=6$, there is a counter-example which is an unmixed matroidal ideal but it is not Cohen-Macaulay.
For $n=4$, we can consider $I=(x_1x_3,x_1x_4,x_2x_3,x_2x_4)$ such that $I$ is an unmixed matroidal ideal but it is not Cohen-Macaulay.

\begin{Corollary}\label{C2}
Let $I$ be an unmixed matroidal ideal of degree $d=2$. Then $|\Ass(R/I)|=m$. In particular, if $I$ is a square-free Veronese ideal, then  $|\Ass(R/I)|=n$.
\end{Corollary}

\begin{proof}
By using Theorem \ref{T1} and Corollary \ref{C1}, it is clear $|\Ass(R/I)|=m$.
If $I$ is a square-free Veronese ideal of degree $d=2$, then $\height(I)=\pd(R/I)=n-1$. Thus by Corollary \ref{C1} we have $m=n$. This completes the proof.
\end{proof}

\begin{Theorem}\label{T2}
Let $I$ be a matroidal ideal of degree $d\geq 2$. Then $I$ is an unmixed ideal if and only if  $(I:x_i)$ is unmixed and $\height(I)=\height(I:x_i)$ for all $1\leq i\leq n$.
\end{Theorem}

\begin{proof}
$(\Longrightarrow)$ let $I$ be an unmixed matroidal ideal. By using \cite[Corollary 1.2]{BH} $(I:x_i)$ is a matroidal ideal.
From the exact sequence $$0\longrightarrow R/(I:x_i)\overset{x_i}\longrightarrow R/I\longrightarrow R/(I,x_i)\longrightarrow 0,$$
we have $\Ass(R/(I:x_i))\subseteq\Ass(R/I)$. Since $I\subseteq (I:x_i)$, it follows that $(I:x_i)$ is unmixed and $\height(I)=\height(I:x_i)$ for all $1\leq i\leq n$.\\
$(\Longleftarrow)$ let $\frak{p}\in\Ass(R/I)$.  By using \cite[Lemma 11]{T}, we have $\Ass(R/I)=\cup_{i=1}^n\Ass(R/(I:x_i))$. Then $\frak{p}\in\Ass(R/(I:x_i))$ for some $i$ and so $\height(I:x_i)=\height(\frak{p})$. Therefore $\height(I)=\height(\frak{p})$ and so $I$ is an unmixed ideal, as required.
\end{proof}

\begin{Example}
Let $n=5$ and $I=(x_1x_3,x_1x_4,x_1x_5,x_2x_3,x_2x_4,x_2x_5)$. Then $I$ is a matroidal ideal such that all $(I:x_i)$ is an unmixed ideal for $1\leq i\leq 5$ but $I$ is not unmixid.
\end{Example}

The following result extends \cite[Lemma 2.2]{BJ}.
\begin{Proposition}\label{P1}
Let $I$ be a polymatroidal ideal of degree $d=2$. If $\Ass(R/I)=\min\Ass(R/I)$, then $I$ is a matroidal ideal or $I=\frak{m}^2$.
\end{Proposition}

\begin{proof}
If $\frak{m}\in\Ass(R/I)$, then $\sqrt{I}=\frak{m}$. Thus by using \cite[Lemma 2.2]{BH} we have $I=\frak{m}^2$.
If $\frak{m}\notin\Ass(R/I)$, then $I=(I:\frak{m})$ and so $I=\cap_{i=1}^n(I:x_i)$. By \cite[Theorem 1.1]{BH}, $(I:x_i)$ is polymatroidal of degree $d=1$. Therefore all $(I:x_i)$ are square-free and so $I$ is square-free. This completes the proof.
\end{proof}

\begin{Example}
Let $n=3$ and $I=(x_1x_2,x_1x_3)$. Then $I$ is a matroidal ideal of degree $2$ such that $\Ass(R/I)=\min\Ass(R/I)$ but $I$ is not unmixed.
\end{Example}

\begin{Example}
Let $n=3$ and $I=(x_1^2x_2,x_1^2x_3)$. Then $I$ is a polymatroidal ideal of degree $d=3$ such that $\Ass(R/I)=\min\Ass(R/I)$ but neither $I$ is matroidal and nor $I=\frak{m}^3$.
\end{Example}
\section{Arithmetical rank of matroidal ideals}

We start this section by following lemma which is proved by Schmitt and Vogel in \cite{SV}.

\begin{Lemma}\label{L2}
Let $P$ be a finite subset of $R$ and let $P_0,P_1,...,P_r$ be subsets of $P$ such that the following conditions hold:
\begin{itemize}
\item[(a)] $\cup_{i=0}^rP_i=P$;
\item[(b)] $P_0$ has exactly one element;
\item[(c)] If $p$ and $p''$ are different elements of $P_i$ ($0<i\leq r$), then there is an integer $j$ with $0\leq j<i$ and an element $p'\in P_j$ such that $pp''\in(p')$.
\end{itemize}
Let $q_i=\sum_{p\in P_i}p$. Then $\sqrt{(P)}=\sqrt{(q_0,...,q_r)}$.
\end{Lemma}

\begin{Lemma}\label{L3}
Let $I$ be a square-free monomial ideal of $R$. If $d=\min\{\deg(u): u\in G(I)\}\geq 2$, then $ara(I)\leq {n-d+1}$.
\end{Lemma}

\begin{proof}
Let $M_i=\{u\in I:$ \ \  $u$ is a square-free element and $\deg(u)=d+i-1\}$ for all $i=1,2,...,n-d+1$. Then $M_{n-d+1}=\{x_1...x_n\}$.
Now, we put $P_{n-d+1-i}=M_{i}$ for all $i=1,2,...,n-d+1$. Since $\sqrt{I}=\sqrt{(\cup_{j=0}^{n-d}P_j)}$, by using Lemma \ref{L2} we have $ara(I)\leq n-d+1$.
\end{proof}

The following result answer to the conjecture that made by H. J. Chiang-Hsieh in \cite[Conjecture]{C}.
\begin{Theorem}\label{T3}
Let $I$ be a mtroidal ideal of degree $d$. Then $ara(I)=n-d+1$.
\end{Theorem}

\begin{proof}
By Lemma \ref{L3} and \cite{L}, we have $\pd(R/I)\leq ara(I)\leq {n-d+1}$.
Since $I$ is a matroidal ideal of degree $d$, we have $\pd(R/I)=n-d+1$. Therefore $ara(I)=n-d+1$, as required.
\end{proof}

\begin{Corollary}
Let $I$ be a matroidal ideal of degree $d$. Then $I$ is square-free Veronese if and only if $I$ is set-theoretic complete intersection.
\end{Corollary}

\begin{proof}
If $I$ is set-theoretic complete intersection, then $I$ is Cohen-Macaulay and so by \cite[Theorem 4.2]{HH2} $I$ is square-free Veronese.
Conversely, let $I$ be square-free Veronese. Then $I$ is Cohen-Macaulay and so $\height(I)=\pd(R/I)$. By Theorem \ref{T3}, we have $\height(I)=ara(I)$, as required.
\end{proof}

\subsection*{Acknowledgements}
We would like to thank deeply grateful to Professors Margherita Barile and Kyouko Kimura for useful discussions.


\end{document}